\documentclass{amsart}
\usepackage{amsmath, amstext, amsbsy, amssymb}

\hoffset \voffset \oddsidemargin=55pt \evensidemargin=55pt
\numberwithin{equation}{section}
\def\beq{\begin{eqnarray}}
\def\eeq{\end{eqnarray}}
\def\beqs{\begin{eqnarray*}}
\def\eeqs{\end{eqnarray*}}
\def\NN{{\mathbb N}}
\def\mz{{\mathbb Z}}
\def\mr{{\mathbb R}}

\def\ind{\hbox{\rm ind}}
\def\ord{\hbox{\rm ord}}

\newfont{\df}{eufm10}

\title[Minimal zero-sum sequences of length four]
{Minimal zero-sum sequences of length four over cyclic group with order $n=p^\alpha q^\beta$}
\thanks{$^\dag$the corresponding author's email: xialimeng@ujs.edu.cn}
\thanks{Supported by the  NNSF of China (Grant No. 11001110, 11271131)}
\author[L.-M. Xia]{Li-meng   Xia$^\dag$}
\author[C.-X. Shen]{Caixia   Shen}

\date{}
\begin{document}
\maketitle
\centerline{Faculty of Science, Jiangsu University}
\centerline{Zhenjiang, 212013, Jiangsu Province, P.R. China}

\def\abstractname{ABSTRACT}
\begin{abstract}
Let $G$ be a finite cyclic group. Every sequence $S$ over $G$ can be written in the form
$S=(n_1g)\cdot...\cdot(n_kg)$ where $g\in G$ and $n_1,\cdots,n_k\in[1,{\hbox{\rm ord}}(g)]$, and the index $\ind S$ of $S$ is defined to be the minimum of $(n_1+\cdots+n_k)/\hbox{\rm ord}(g)$ over all possible $g\in G$ such that $\langle g\rangle=G$. A conjecture says that if $G$ is finite such that $\gcd(|G|,6)=1$, then $\ind(S)=1$ for every minimal zero-sum sequence $S$. In this paper, we prove that the conjecture holds if $|G|$ has two prime factors.


\end{abstract}

\newtheorem{theo}{Theorem}[section]
\newtheorem{theorem}[theo]{Theorem}
\newtheorem{defi}[theo]{Definition}
\newtheorem{conj}[theo]{Conjecture}
\newtheorem{lemma}[theo]{Lemma}
\newtheorem{coro}[theo]{Corollary}
\newtheorem{proposition}[theo]{Proposition}
\newtheorem{remark}[theo]{Remark}

\setcounter{section}{0}

\section{Introduction}

Throughout the paper, let $G$ be an additively written finite cyclic group of order $|G| = n$. By
a sequence over $G$ we mean a finite sequence of terms from $G$ which is unordered and repetition
of terms is allowed. We view sequences over $G$ as elements of the free abelian monoid $\mathcal{F}(G)$
and use multiplicative notation. Thus a sequence $S$ of length $|S| = k$ is written in the form
$S = (n_1g)\cdot...\cdot(n_kg)$, where $n_1,\cdots,n_k\in{\mathbb N}$ and $g\in G$. We call $S$ a {\it zero-sum sequence} if $\sum^k_{j=1}n_jg = 0$. If $S$ is a zero-sum sequence, but no proper nontrivial subsequence of $S$ has sum zero, then $S$ is called a {\it minimal zero-sum sequence}. Recall that the index of a sequence $S$ over $G$
is defined as follows.

\begin{defi}
For a sequence over $G$
\beqs S=(n_1g)\cdot...\cdot(n_kg), &&\hbox{where}\;1\leq n_1,\cdots,n_k\leq n,\eeqs
the index of $S$ is defined  by $\ind(S)=\min\{\|S\|_g|g\in G \hbox{~with~}\langle g\rangle=G\}$, where
\beq \|S\|_g=\frac{n_1+\cdots+n_k}{\ord(g)}.\eeq
\end{defi}
Clearly, $S$ has sum zero if and only if $\ind(S)$ is an integer.

\begin{conj}
Let $G$ be a finite cyclic group such that $\gcd(|G|,6)=1$. Then every minimal zero-sum sequence $S$ over $G$ of length $|S|=4$ has $\ind(S)=1$.
\end{conj}

The index of a sequence is a crucial invariant in the investigation of (minimal) zero-sum
sequences (resp. of zero-sum free sequences) over cyclic groups. It was first addressed by
Kleitman-Lemke (in the conjecture [9, page 344]), used as a key tool by Geroldinger ([6, page736]), and then investigated by Gao [3] in a systematical way. Since then it has received a great
deal of attention (see for example [1, 2, 4, 7, 10, 11, 12, 13, 14, 15, 16, 17, 18]). A main focus of the investigation of index is to determine minimal zero-sum sequences of index 1. If $S$ is a minimal zero-sum sequence of length $|S|$ such that $|S|\leq3$ or $|S|\geq\lfloor \frac{n}2\rfloor+2$, then $\ind(S)=1$ (see [1, 14, 16]). In contrast to that, it was shown that for each $k$ with $5\leq k\leq \lfloor \frac{n}2\rfloor+1$, there is a minimal zero-sum subsequence $T$ of length $|T| = k$ with $\ind(T)\geq 2$ ([13, 15]) and that the same is true for $k = 4$ and $\gcd(n, 6)\not= 1$ ([13]). The left case leads to the above conjecture.

In [12], it was proved that Conjecture 1.2 holds true if $n$ is a prime power. In [11], it was proved that Conjecture 1.2 holds for $n=p_1^\alpha\cdot p_2^\beta$, $(p_1\not=p_2)$, and at
least one $n_i$ co-prime to $|G|$.  However, the general case is still open. In [19], it was proved that Conjecture 1.2 holds if the sequence $S$ is reduced and at least one $n_i$ co-prime to $|G|$.

In this paper, we give the affirmative proof of Conjecture 1.2 for general case under assumption $n=p^\alpha q^\beta$.

\begin{theo}
Let $G$ be a  finite cyclic group of order $|G|=p^\alpha q^\beta$, where $\alpha,\beta\in\NN$, and $p,q$ are distinct primes, such that  $\gcd(|G|,6)=1$. Then every minimal zero-sum sequence $S$ over $G$ of length $|S|=4$ has $\ind(S)=1$.
\end{theo}

It was mentioned in [13] that Conjecture 1.2 was confirmed computationally if $n\leq1000$.
Hence, throughout the paper, we always assume that $n>1000$.

\section{Reduction to a special case}

Given real numbers $a, b\in\mr$,
we use $[a, b] = \{x\in\mz|a\leq x\leq b\}$ to denote the set of integers between $a$ and $b$, and similarly, we set $[a, b) = \{x\in\mz|a\leq x < b\}$. For $x\in\mz$,
we denote by $|x|_n\in[1, n]$ the integer congruent to $x$ modulo $n$.

Throughout this paper, let $G$ be a finite cyclic group of order $|G|=n=p^\alpha q^\beta>1000$, where $\alpha,\beta\in\NN$ and $p,q$ are distinct primes greater than or equal to $5$.

First we show that Theorem 1.3 can be reduced to sequences of a special form.

\begin{proposition}
Let $S=(eg)\cdot(cg)\cdot((n-b)g)\cdot((n-a)g)$ be a minimal zero-sum sequence over $G$, where $g\in G$ with order $\ord(g)=|G|=p^\alpha q^\beta$ and $e,a,b,c\in [1,n-1]$ such that $e<a\leq b<c<\frac{n}{2}$ and $e+c=a+b$. Then $\ind(S)=1$.
\end{proposition}

\begin{proof}\textsf{Proof of Theorem 1.3 based on Proposition 2.1.}  Let $S=(n_1g)\cdot(n_2g)\cdot(n_3g)\cdot(n_4g)$ where $g\in G$ with $\ord(g)=|G|$ and $n_1, n_2, n_3, n_4\in [1,n-1]$. Now do the reduction to the special case in Proposition 2.1.

Notice the following two sufficient conditions (introduced in Remark 2.1 of [11]):

(1) If there exists positive integer $m$ such that $\gcd(n,m)=1$ and $|mn_1|_n+|mn_2|_n+|mn_3|_n+|mn_4|_n=3n$, then $\ind(S)=1$.

(2) If there exists positive integer $m$ such that $\gcd(n,m)=1$ and at most one $|mn_i|_n\in\left[1,\frac{n}{2}\right]$ (or, similarly, at most one $|mn_i|_n\in\left[\frac{n}{2},n\right]$), then $\ind(S)=1$.

Hence we can assume that $n_1+n_2+n_3+n_4=2n$ and $n_1\leq n_2<\frac{n}{2}<n_3\leq n_4$. By the minimality of $S$, it doesn't hold $n_1+n_4=n$. Next we may assume that $n_1+n_4<n$. Otherwise we let $m=n-1$ and  consider the sequence $$(n_1',n_2',n_3',n_4')=(|mn_4|_n, |mn_3|_n, |mn_2|_n, |mn_1|_n)=(n-n_4,n-n_3,n-n_2,n-n_1).$$

Let $e=n_1, c=n_2, b=n-n_3$ and $a=n-n_4$, then $e<a\leq b<c<\frac{n}{2}$ and $n_1+n_2+n_3+n_4=2n$ implies that $e+c=a+b$.
\end{proof}

Proposition 2.1 is already well-known in some special cases. The following three lemmas are analogues of Lemma 2.3, Lemma 2.5 and Lemma 2.6 in [11], and the proof is very similar.

\begin{lemma} Proposition 2.1 holds if one of the following conditions holds :

(1) There exist positive integers $k, m$ such that $\frac{kn}{c}\leq m \leq\frac{kn}{b}$, $\gcd(m, n) = 1$,  $1\leq k\leq b$ and
$ma < n$.

(2) There exists a positive integer $M\in[1,\frac{n}{2e}]$
such that $\gcd(M, n) = 1$ and at least two of the
following inequalities hold :
\beqs |Ma|_n >\frac{n}2, |Mb|_n >\frac{n}2, |Mc|_n < \frac{n}2.\eeqs
\end{lemma}

\begin{lemma} Suppose $s\geq2$, $a>2e$ and $[\frac{(2s-2t-1)n}{2b}, \frac{(s-t)n}{b}]$ contains an integer co-prime to $n$ for some $t\in [0, \lfloor\frac{s}{2}\rfloor-1]$. Then Proposition 2.1 holds.
\end{lemma}

\begin{lemma} Suppose $s\geq2$, $a>2e$ and $[\frac{(2s-2t-1)n}{2b}, \frac{(s-t)n}{b}]$ contains no integers co-prime to n for every $t\in [0, \lfloor\frac{s}{2}\rfloor-1]$. Then the following results hold.

(i)  $\frac{n}{2b} < 3$ (where $\frac{n}{2b}$ is the length of the interval $[\frac{(2s-2t-1)n}{2b}, \frac{(s-t)n}{b}]$ for each $t\in [0, \lfloor\frac{s}{2}\rfloor-1]$).

(ii)  If $s\geq 4$, then $[\frac{(2s-2t-1)n}{2b}, \frac{(s-t)n}{b}]$ contains exactly one integer for every $t\in [0, \lfloor\frac{s}{2}\rfloor-1]$. Furthermore, $\frac{n}{2b} < 2$.

(iii) Suppose that $s\geq4$, $x\in [\frac{(2s-2t-1)n}{2b}, \frac{(s-t)n}{b}]$ and $y\in[\frac{(2s-2t-3)n}{2b}, \frac{(s-t-1)n}{b}]$ for some $t\in [0, \lfloor\frac{s}{2}\rfloor-2]$. Then $\gcd(x, y, n) = 1$.

(iv) Suppose that $s\geq 6$, $x\in [\frac{(2s-2t-1)n}{2b}, \frac{(s-t)n}{b}]$ and $z\in[\frac{(2s-2t-5)n}{2b}, \frac{(s-t-2)n}{b}]$ for some $t\in [0, \lfloor\frac{s}{2}\rfloor-3]$. Then $\gcd(x, z, n) > 1$ and $5|\gcd(x, z, n)$. Furthermore, $z = x-5$ and $\frac{n}{2b} < \frac{7}{5}$.

(v) $s\leq 7$.
\end{lemma}

Next we show that a further reduction of parameters can be done. Let $$S=(eg)\cdot(cg)\cdot((n-b)g)\cdot((n-a)g)=(n_1g)\cdot(n_2g)\cdot(n_3g)\cdot(n_4g),$$
where $e,a,b,c$ and $g$ are as in Proposition 2.1 and $n_1=e, n_2=c, n_3=n-b$ and $n_4=n-a$.

Let $u$ be the greatest common divisor of $n, n_1, n_2, n_3, n_4$. If $u>1$, we can consider $G'=\langle ug\rangle$ and $S= (\frac{n_1}{u}ug)\cdot(\frac{n_2}{u}ug)\cdot(\frac{n_3}{u}ug)\cdot(\frac{n_4}{u}ug)$, where $|G'|=\frac{n}{u}$ is less than $n$.  Hence we can assume that $u=1$. By the result of [11], we can assume that $\gcd(n_i,n)>1$ for $i=1,2,3,4$. Clearly, under this assumption, two of $n_i$'s have factor $p$ and the other two have factor $q$.

We define $i_0$ and $j_0$ by
\beq p^{i_0}=\min\Big\{\gcd(n_i,n)\Big| p|n_i, i\in[1,4]\Big\}&&q^{j_0}=\min\Big\{\gcd(n_i,n)\Big|  q|n_i, i\in[1,4]\Big\},\eeq
such that  $p^{i_0}<q^{j_0}$.

\begin{proposition} It is sufficient to prove Proposition 2.1 under the following parameters:

(1) $n\geq75p^{i_0}$;

(2) $e\in\{p^{i_0},q^{j_0},2q^{j_0}\}$ and $a>3e$;

(3) If $e\in\{q^{j_0},2q^{j_0}\}$, then $a\geq 6e$;

(4) $s\leq7$.
\end{proposition}

\begin{proof}
If $i_0=\alpha$ and $j_0=\beta$, without less of generality, let $p|n_1, p|n_2$, then the sum of $p^{\alpha}|(n_1+n_2)$ and $q^\beta|(\nu n-n_3-n_4)=(n_1+n_2)$, hence $n|(n_1+n_2)$, which contradicts to that $S$ is a minimal zero-sum sequence. Then we infer that $\alpha+\beta>i_0+j_0$ and $\frac{n}{p^{i_0}}\geq5q^{j_0}>5p^{i_0}$. If $p^{i_0}\geq15$, then $n\geq75p^{i_0}$. Otherwise, we have $p^{i_0}\leq13$ and $\frac{n}{p^{i_0}}\geq\frac{1000}{13}>75$.

Now we renumber the sequence such that $e<\frac{a}{3}$. First we may assume that $e=p^{i_0}$. Then, for the purpose, we only need to consider the following three situations.

{\noindent\bf The first situation:} $2e>a$, then $a=q^{j_0}$.

{\bf Case 1.} $a|b$.

Let $m=\frac{n+a}{a}$,$m_1=\frac{n+2a}{a}$, $m_2=\frac{n+3a}{a}$, $m_3=\frac{n+4a}{a}$.

If $\gcd(n,m)=1$ then
\beqs &&|me|_n>\frac{n}{2},\quad \hbox{\rm since\;}\frac{n+a}{2}<\frac{n+a}{a}e\leq\frac{n+a}{a}(a-2)<\frac{5n}{7}+a-1<n,\\
&&|m(n-a)|_n=n-a>\frac{n}{2},|m(n-b)|_n=n-b>\frac{n}{2}.\eeqs

If $\gcd(n,m)>1$, then $j_0=\beta$ and $\gcd(n,m_1)=\gcd(n,m_2)=\gcd(n,m_3)=1$. Moreover,
\beqs |m_1e|_n>\frac{n}{2}, |m_2e|_n>\frac{n}{2},|m_3e|_n>\frac{n}{2},|m_1a|_n<\frac{n}{2},|m_2a|_n<\frac{n}{2}, |m_3a|_n<\frac{n}{2}.\eeqs

If $b<\frac{n}{4}$, we have $|m_1(n-b)|_n=n-2b>\frac{n}{2}$. If $\frac{n}{4}<b<\frac{n}{3}$, we have $|m_3(n-b)|_n=2n-4b>\frac{n}{2}$. If $\frac{n}{3}<b<\frac{n}{2}$, we have $|m_2(n-b)|_n=2n-3b>\frac{n}{2}$.
Then we can find an integer $m_i$ such that $\gcd(n,m_i)=1$ and all of $|m_ie|_n, |m_i(n-b)|_n, |m_i(n-a)|_n$ are larger than $\frac{n}{2}$, which implies that $\ind(S)=1$.

{\bf Case 2.} $a|c$.

Let $m=\frac{n-a}{a}$, $m_1=\frac{n-2a}{a}$, $m_2=\frac{n+3a}{2a}$, $m_3=\frac{n+5a}{2a}$.

If $\gcd(n,m)=1$, then $\frac{n}{2}<|me|_n<n-10a$ and $|mc|_n=n-c>\frac{n}{2}$. For this case, if $|m(n-b)|_n>\frac{n}{2}$, we have done. Otherwise, it must hold $a<|m(n-b)|_n$. We get a renumbering:
\beq e'=a, c'=|m(n-b)|_n, \{b',a'\}=\{c,n-|me|_n\},\eeq
and it is easy to check that $a'\geq6e'$.

If $\gcd(m,n)>1$, then $a=q^{\beta}$, $q|(p^\alpha-1)$ and $\gcd(n,m_1)=\gcd(n,m_2)=\gcd(n,m_3)=1$.

{\it Subcase 1.} $c=2ta$ for some integer $t$.

Let $m=\frac{n+a}{2a}$. Then $|me|_n<\frac{n}{2}$, $|mc|_n=\frac{c}{2}<\frac{n}{2}$, $|m(n-a)|_n=\frac{n-a}{2}<\frac{n}{2}$.

{\it Subcase 2.} $c=(2t+1)a$ for some integer $t$.

If $\frac{n}{4}>c$, replace $m$ by $m_1$ and repeat the above process, we have $|m_1(n-b)|_n>\frac{n}{2}$,
$|m_1c|_n>\frac{n}{2}$ and $|m_1e|_n>\frac{n}{2}$, which implies $\ind(S)=1$, or we can obtain a renumbering:
\beq e'=2a, c'=|m_1(n-b)|_n, \{b',a'\}=\{2c,n-|m_1e|_n\},\eeq
it also holds that $a'\geq 6e'$.

If $\frac{n}{4}<c<\frac{n}{3}$, $|m_3a|_n=\frac{n-5a}{2}<\frac{n}{2}$. We have $|m_3e|_n<\frac{n}{2}$ and $|m_3c|_n=|\frac{n+5c}{2}|_n<\frac{n}{2}$, exactly it belongs to $(\frac{n}{8},\frac{n}{3})$. Then $\ind(S)=1$.

If $\frac{n}{3}<c$, $|m_2a|_n=\frac{n-3a}{2}<\frac{n}{2}$. We have $|m_2c|_n=|\frac{n+3c}{2}|_n<\frac{n}{4}$, $|m_2e|_n<\frac{n}{2}$, and hence $\ind(S)=1$.

{\noindent\bf The second  situation:} $2e<a<3e$ and $e|b$.

Let $b=te$, we have $\frac{tn}{b}=\frac{n}{e}$. Then $\frac{tn}{b}-\frac{tn}{c}=\frac{t(a-e)n}{bc}>\frac{bn}{bc}>2$, and at least two integers $m_1=\frac{tn}{b}-1=\frac{n-e}{e}, m_2=m_1-1=\frac{n-2e}{e}$ contained in $(\frac{tn}{c},\frac{tn}{b})$.

It is easy to see that at least one of $m_1,m_2$ is co-prime to $n$. Let $m$ be one of them such that $\gcd(m,n)=1$.
Then we have $me<n$, $mc\geq tn$, $tn>mb\geq m_1b=t(n-2e)=tn-2b>(t-1)n$ and $2n<2n-4e+\frac{n}{e}-2=\frac{n-2e}{e}(2e+1)\leq ma<3n$. Hence
\beqs 3n&\geq& |me|_n+|mc|_n+|m(n-b)|_n+|m(n-a)|_n\\
&\geq& me+(mc-tn)+(tn-mb)+(3n-ma)=3n.\eeqs
Where $-4e+\frac{n}{e}-2>0$ because $n\geq\min\{\frac{pea}2,\frac{qea}2\}\geq\frac{5ea}2>5e^2$ and $\frac{n}{e}>5e>4e+5$. Thus $\ind(S)=1$.

{\noindent\bf The third  situation:} $2e<a<3e$ and $e|c$.

{\bf Case 1.} $a=q^{j_0}$ and $b=(2t+1)a$ .

Let $m=\frac{n-a}{a}$. If $\gcd(n,m)=1$, then $|me|_n<\frac{n}{2}$, $|m(n-a)|_n=a<\frac{n}{2}$, $|m(n-b)|_n=b<\frac{n}{2}$. We have done.

If $\gcd(n,m)>1$, let $m_1=\frac{n+a}{2a}$, then $\gcd(n,m_1)=1$. $|m_1e|_n<\frac{n}{2}$, $|m_1(n-a)|_n=\frac{n-a}{2}<\frac{n}{2}$, $|m_1(n-b)|_n=\frac{n-b}{2}<\frac{n}{2}$, hence $\ind(S)=1$.

{\bf Case 2.} $a=q^{j_0}$ and $b=2tq^{j_0}$.

Let $m=\frac{n-a}{a}$. If $\gcd(n,m)=1$, then $|me|_n<\frac{n}{2}$, $|m(n-a)|_n=a<\frac{n}{2}$, $|m(n-b)|_n=b<\frac{n}{2}$. We have done.

If $\gcd(n,m)>1$, let $m_1=\frac{n-2a}{a}$,$m_2=\frac{n+3a}{2a}$,$m_3=\frac{n+a}{2a}$.

If $b<\frac{n}{4}$, then $|m_1e|_n<\frac{n}{2}$, $|m_1(n-a)|_n=2a<2b<\frac{n}{2}$, $|m(n-b)|_n=2b<\frac{n}{2}$. We have done.

If $\frac{n}{4}<b<\frac{n}{3}$, then $|m_3e|_n<\frac{n}{2}$, $|m_3(n-a)|_n=\frac{n-a}2<\frac{n}{2}$, $|m_3(n-b)|_n=\frac{b}{2}<\frac{n}{2}$. We have done.

If $\frac{n}{3}<b<\frac{n}{2}$, then $|m_2e|_n<\frac{n}{2}$, $|m_2(n-a)|_n=\frac{n-3a}2<\frac{n}{2}$, $|m_2(n-b)|_n<\frac{n}{2}$. We have done.

{\bf Case 3.}
$a=2q^{j_0}$ and $b=2tq^{j_0}$.

Let $m=\frac{n-q^{j_0}}{2q^{j_0}}$. If $\gcd(n,m)=1$, then $|me|_n<\frac{n}{2}$, $|m(n-a)|_n=q^{j_0}<\frac{n}{2}$, $|m(n-b)|_n=tq^{j_0}<\frac{n}{2}$. We have done.

If $\gcd(n,m)>1$, let $m_1=\frac{3n-q^{j_0}}{2q^{j_0}}$.

Then $|m_1(n-a)|=\frac{a}{2}<\frac{n}{2}$, $|m_1(n-b)|=\frac{b}{2}<\frac{n}{2}$,
and $|m_1e|<m_1e-n<\frac{n}2$, hence $\ind(S)=1$.

{\bf Case 4.} $a=2q^{j_0}$ and $b=(2t+1)q^{j_0}$.

Let $m=\frac{n-q^{j_0}}{2q^{j_0}}$. If $\gcd(n,m)=1$, then $|me|_n<\frac{n}{2}$, $|m(n-a)|_n=q^{j_0}<\frac{n}{2}$, $|m(n-b)|_n=n-tq^{j_0}>\frac{n}{2}$. Clearly, $t\geq4$.

We also have $|mc|_n\in(\frac{c}{2},\frac{n+c}{2})$. If $|mc|_n<\frac{n}{2}$, then we have done.
If $|mc|_n>\frac{n}{2}$, then $n-|mc|_n>\frac{n-c}2\geq10q^{j_0}$, and we have renumbering
\beq e'=q^{j_0}, c'=|me|_n, \{b', a'\}=\{|mb|_n, n-|mc|_n\}, \quad e'<a'\leq b'<c'<\frac{n}{2}.\eeq
Moreover, if $p^{i_0}|(e'-a')$, we have $a'\geq 6e'$. Then it always holds that $a'\geq 6e'$ after this renumbering.

Up to now, we finish the renumbering. Hence, we can always assume that $e\in\{p^{i_0},q^{j_0},2q^{j_0}\}$ and $a>3e$.
Particularly, $a\geq 6e$ when $e\in\{q^{j_0},2q^{j_0}\}$.  Then in view of Lemmas 2.2,
2.3 and 2.4 and the above renumbering, from now on we may always assume that $s\leq7$.
\end{proof}

Let $k_1$ be the largest positive integer such that $\lceil\frac{(k_1-1)n}{c}\rceil=\lceil\frac{(k_1-1)n}{b}\rceil$ and $\frac{k_1n}{c}\leq  m < \frac{k_1n}{b}$.
The existence of integer $k_1$ has been proved in [11].

As mentioned above, we only need prove Proposition 2.1 under the parameters listed in Proposition 2.5. We now show that Proposition 2.1 holds through the following 3 propositions.

\begin{proposition} Suppose $\lceil\frac{n}{c}\rceil<\lceil\frac{n}{b}\rceil$, then Proposition 2.1 holds under the parameters listed in Proposition 2.5. \end{proposition}
\begin{proposition} Suppose $\lceil\frac{n}{c}\rceil=\lceil\frac{n}{b}\rceil$. Let $k_1$ be the largest positive integer such that $\lceil\frac{(k_1-1)n}{c}\rceil=\lceil\frac{(k_1-1)n}{b}\rceil$ and $\frac{k_1n}{c}\leq  m_1 < \frac{k_1n}{b}$ holds for some integer $m_1$. If $k_1>\frac{b}{a}$, then Proposition 2.1 holds under the parameters listed in Proposition 2.5. \end{proposition}
\begin{proposition} Suppose $\lceil\frac{n}{c}\rceil=\lceil\frac{n}{b}\rceil$. Let $k_1$ be the largest positive integer such that $\lceil\frac{(k_1-1)n}{c}\rceil=\lceil\frac{(k_1-1)n}{b}\rceil$ and $\frac{k_1n}{c}\leq  m_1 < \frac{k_1n}{b}$ holds for some integer $m_1$. If $k_1\leq\frac{b}{a}$, then Proposition 2.1 holds under the parameters listed in Proposition 2.5. \end{proposition}

\section{Proof of Proposition 2.6}

In this section, we assume that $\lceil\frac{n}{c}\rceil<\lceil\frac{n}{b}\rceil$. Let $m_1=\lceil\frac{n}{c}\rceil$. Then we have $m_1-1 <\frac{n}c\leq m_1 < \frac{n}{b}$. By Lemma 2.3 (1), it suffices to $m$ and $k$ such that $\frac{kn}{c}\leq m <\frac{kn}{b}$, $\gcd(m, n) =1$, $1\leq k\leq b$, and $ma < n$. So in what follows, we may always assume that $\gcd(n, m_1)> 1$.

\begin{lemma}
Let $e,a,b,c$ be parameters listed in Proposition 2.5. We have the following estimates:

(1) If $35|n$, then $n>71e$;

(2) If $35|n$, then $n\geq 125e$ or $a\geq 11e$;

(3) If $55|n$, then $n\geq125e$;

(4) If $5|n$ and $\gcd(77,n)=1$, then $a\geq 125e$ for $e=p^{i_0}$ and $a\geq 25e$ for $e\in\{q^{j_0},2q^{j_0}\}$.
\end{lemma}

This lemma can be showed simply and we omit the proof.

\begin{lemma} If $\left[\frac{n}{c}, \frac{n}{b}\right]$ contains at least two integers, then ind(S) = 1.\end{lemma}
\begin{proof}
The proof is similar to that of Lemma 3.4 in [11].\end{proof}

By Lemma 3.2, we may assume that $\left[\frac{n}{c}, \frac{n}{b}\right]$ contains exactly one integer $m_1$, and thus
\beq m_1-1 <\frac{n}{c}\leq m_1 <\frac{n}{b}<m_1+1.\eeq

Let $l$ be the smallest integer such that $\left[\frac{ln}{c}, \frac{ln}{b}\right)$ contains at least three integers. Clearly, $l\geq2$. We claim that it holds either(referred to [11])
\beq lm_1-2<\frac{ln}{c}<\frac{ln}{b}<lm_1+3\eeq
or
\beq lm_1-3<\frac{ln}{c}<\frac{ln}{b}<lm_1+2.\eeq

\begin{lemma}
Assume that \beqs
&5l-2 < \frac{ln}{c} < 5l-1 < 5l < 5l+1 < \frac{ln}b < 5+2,\\
&5(l-1)-1 < \frac{(l-1)n}{c} < 5(l-1) < 5(l-1) + 1 < \frac{(l-1)n}b < 5(l-1) + 2,\eeqs
and $\gcd(5l-1, n) = 1$, $l\in[3,9]$, $5|n$. Then $\ind(S)=1$.
\end{lemma}
\begin{proof}
It is sufficient to show that $ma<n$ for $m=5l-1$.

If $e\leq \frac{a}{5}$, then
$$ma = (5l-1)(c-b + e) < \frac{5}{4}(5l-1)\left(\frac{ln}{5l-2}-\frac{ln}{5l+2}\right)=\frac{(100l^2-20l)n}{100l^2-16}<n,$$
and we have done.

Next we can assume that $e=p^{i_0}>\frac{a}{5}$. It is easy to know that $a\in\{q^{j_0},2q^{j_0},3q^{j_0},4q^{j_0}\}$.

{\bf Case 1.} $5|e$.

If $e=5$, then $a\in\{17, 19, 21, 22, 23\}$. When $a\in\{17, 19, 23\}$, we have $\frac{n}{a}\geq5q\geq 85>(5l-1)$ and we have done.

Moreover, we have $\frac{n}{a}\geq\frac{1375}{22}>62$ for $a=22$ and $\frac{n}{a}\geq\frac{1225}{21}>58$ for $a=21$, both of them contradict to $a>\frac{b}{8}>\frac{n}{48}$.

If $e\geq 125$, we have $n>625e$. Then
$$ma = (5l-1)(c-b + e) <(5l-1)\left(\frac{ln}{5l-2}-\frac{ln}{5l+2}+e\right)=\frac{(20l^2-4l)n}{25l^2-4}+(5l-1)e<n,$$
we have done.

Let $e=25$.
If $n\not=125q^{j_0}$, we have $n\geq25q^{j_0}\geq25\times 29=725$. If $q^{j_0}\geq 67$, we have $n\geq 635e$. Both of these two situations imply that
$$ma<(5l-1)\left(\frac{ln}{5l-2}-\frac{ln}{5l+2}+e\right)=\frac{(20l^2-4l)n}{25l^2-4}+(5l-1)e<n.$$
Then we have done.

Let $n=125q^{j_0}$. If $a\leq2q^{j_0}$ and $n\geq\frac{125a}{2}>62a$, which contradicts to $a>\frac{b}{8}>\frac{n}{48}$.

If $a=3q^{j_0}$, then $e|c$. Otherwise we have $c\geq28q^{j_0}$ and $b=c+e-a\geq 25q^{j_0}+e>8a$, a contradiction. So $c=25(q^{j_0}-1)$, which implies $\frac{n}{c}>5$, or $c\geq 25(2q^{j_0}-1)$, which implies $b\geq 47q^{j_0}>8a$, both of them give a contradiction.

We infer that $a=4q^{j_0}$, hence $q^{j_0}=q\in\{29,31\}$, similar to the above process, we obtain a contradiction.

{\bf Case 2.} $\gcd(5,e)=1$.

If $e\geq 29$, we have $q^{j_0}\geq 125$ and $n\geq 625e$. Then it is easy to check that $ma<n$. We can assume that $e=p\in\{7,11,13,17,19,23\}$ and $q^{j_0}=25$.

Moreover, we have $c=p\times 24$ or $b=26\times p$(using the condition $s\leq7$), these imply $\frac{n}{c}>5$ or $\frac{n}{b}<5$, a contradiction.
\end{proof}

\begin{lemma}
If $4 < \frac{n}{c}\leq 5 < \frac{n}{b} < 6$ and $5|n$, then $\ind(S) = 1$.\end{lemma}
\begin{proof}
Since $4 < \frac{n}{c}\leq 5 < \frac{n}{b} < 6$, $n > 5b$. Note that $m_1 = \lceil\frac{n}{c}\rceil = 5$.

If $l=2$, since $\left[\frac{ln}{c}, \frac{ln}{b}\right)$ contains at least three integers, we must have $8 < \frac{2n}{c} < 9 < 10 <
11 < \frac{2n}{b} < 12$.  Thus $\frac{n}{6} < b < c < \frac{n}{4}$. Let $m = 9$ and $k = 2$. Then by Proposition 2.5, $9a=9\times(c-b+e)<9\times\left(\frac{n}{4}-\frac{n}{6}+e\right)=\frac{3n}{4}+9e<n$ or
$9a\leq\frac{6}{5}\times9\times(c-b)<\frac{54}{5}\times\left(\frac{n}{4}-\frac{n}{6}\right)=\frac{9n}{10}<n$, and we are done.

Next assume that $l\geq3$. Since $\left[\frac{ln}{c}, \frac{ln}{b}\right)$ contains at least three integers and $5l-3<\frac{ln}{c}<\frac{ln}{b}\leq 5l + 3$, we can divide the proof into three cases.

{\bf Case 1.} $5l+2<\frac{ln}{b}\leq 5l+3$. Then $\frac{2}{l}\leq\frac{n}{b}-5\leq\frac{3}{l}$.

For $\gamma\in[\frac{l+1}{2},l-1]$, since $\gamma(\frac{n}{b}-5)>\frac{l}{2}\cdot\frac{2}{l}=1$ and thus $\frac{\gamma n}{c}\leq 5\gamma<5\gamma+1<\frac{\gamma n}{b}$. By the minimality of $l$ we infer that
\beq 5\gamma-1<\frac{\gamma n}{c}\leq 5\gamma<5\gamma+1<\frac{\gamma n}{b}<5\gamma+2.\eeq
Let $\gamma=l-1$. We have $(5(l-1)-1)(b+a-e)=(5(l-1)-1)c<(l-1)n<(5(l-1)+2)b$  and thus $(5l-6)(a-e)<3b$.

If $l\geq 16$, let $k=l$ and let $m$ be an integer in $\left[\frac{ln}{c}, \frac{ln}{b}\right)$ which is co-prime to $n$. Then $m\leq 5l+2$ and
\beqs ma\leq(5l+2)<\frac{5l+2}{5l-6}\times\frac{3}{2}\times(5l-6)(a-e)<\frac{5\times 16+2}{5\times16-6}\times\frac{3}{2}\times3b<5b\leq n,\eeqs
and we have done.

Next assume that $l\in[6,15]$.

If $\gcd(5l-4,n)=1$, let $m=5l-4$ and $k=l-1$. Then by (3.4) $\frac{kn}{c}\leq m<\frac{kn}{b}$ and
\beqs ma=(5l-4)<\frac{5l-4}{5l-6}\times\frac{3}{2}\times(5l-6)(a-e)<\frac{5\times 6-4}{5\times6-6}\times\frac{3}{2}\times3b<5b\leq n,\eeqs
as desired. Thus we may assume that $\gcd(5l-4,n)>1$.

Applying (3.4) with $\gamma=l-2$, we have $\gcd(5l-9,n)=1$ and $5l-11<\frac{(l-2)n}{c}\leq 5l-10<5l-9<\frac{(l-2)n}{b}\leq 5l-8$. Thus $\frac{(l-2)n}{5l-8}\leq b<c<\frac{(l-2)n}{5l-11}$. Let $m=5l-9$ and $k=l-2$, we have
\beqs ma=(5l-9)a<\frac{3}{2}\times(5l-9)\times\left(\frac{(l-2)n}{5l-11}-\frac{(l-2)n}{5l-8}\right)<n,\eeqs
and we have done.

Finally, assume that $l\leq5$.

If $l\in[4,5]$, applying (3.4) with $\gamma=3$, we have $14<\frac{3n}{c}\leq15<16<\frac{3n}{b}\leq 17$. then
$\frac{3n}{17}\leq b<c<\frac{3n}{14}$. Note that $\gcd(n,16)=1$. Let $m=16$ and $k=3$. Then $$ma=16a<16\times \frac{3}{2}\times\left(\frac{3n}{14}-\frac{3n}{17}\right)=\frac{27\times 16n}{28\times17}<n,$$ and we have done.

If $l=3$, we have $\frac{3n}{c}\leq15<16<17<\frac{3n}{b}\leq 18$. If $\frac{3n}{c}>14$, then $c<\frac{3n}{14}$. Let $k=3$ and $m=16$. By Lemma 3.1,  we have $16a<16\times\frac{11}{10}\times\left(\frac{3n}{14}-\frac{n}{6}\right)=\frac{88n}{105}<n$,
or $16a<16\times\left(\frac{3n}{14}-\frac{n}{6}+\frac{n}{125}\right)<n$, as desired.
If $\frac{3n}{c}\leq14$, we have $13<\frac{3n}{c}\leq 14$. Applying (3.4) with $\gamma=2$, we have $9<\frac{2n}{c}\leq10<11<\frac{2n}{b}\leq 12$, and then
$\frac{n}{6}\leq b<c<\frac{2n}{9}$. Note that either $\gcd(11,n)=1$ or $\gcd(n,14)=1$. Now let $m=11$ and $k=2$ if $\gcd(n,11)=1$, or let $m=14$ and $k=3$ if $\gcd(n,14)=1$. Then $$ma\leq14a<14\times \frac{3}{2}\times\left(\frac{3n}{13}-\frac{1n}{6}\right)=\frac{77n}{78}<n,$$ and we have done.

This completes the proof of Case 1.

{\bf Case 2.} $\frac{ln}{b}\leq 5l+2$ and $5l-3<\frac{ln}{c}\leq  5l-2$. This case can be proved in a similar way to Case 1.

{\bf Case 3.} $\frac{ln}{b}\leq 5l+2$ and $\frac{ln}{c}>5l-2$. Thus $5l-2<\frac{ln}{c}\leq 5l-1<5l<5l+1<\frac{ln}{b}\leq 5l+2$. This implies that every integer in $\left[\frac{ln}{c},\frac{ln}{b}\right)$ is less that $5l+2$. By the minimality of $l$, we must have one of the following holds.

(i) $5l-6<\frac{(l-1)n}{c}\leq 5l-5<\frac{(l-1)n}{b}\leq 5l -4$.

(ii) $5l-6<\frac{(l-1)n}{c}\leq 5l-5<5l-4<\frac{(l-1)n}{b}\leq 5l -3$.

(iii) $5l-7<\frac{(l-1)n}{c}\leq 5l-6<5l-5<\frac{(l-1)n}{b}\leq 5l -4$.

We divide the proof into three subcases according the above three situations.

{\it Subcase 3.1}. (i) holds. Let $k=l$ and $m$ be an integer in $\left[\frac{ln}{c},\frac{ln}{b}\right)$ which is co-prime to $n$. Note that $m\leq 5l+1$, then
\beqs ma\leq(5l+1)a<\frac{3}{2}\times(5l+1)\times\left(\frac{(l-1)n}{5l-6}-\frac{(l-1)n}{5l-4}\right)=\frac{3(l-1)(5l+1)n}{(5l-6)(5l-4)}<n,\eeqs
and we have done.

{\it Subcase 3.2}. (ii) holds.

If $l\geq 10$, then let $k=l$ and $m$ be an integer in $\left[\frac{ln}{c},\frac{ln}{b}\right)$ which is co-prime to $n$. Note that $m\leq 5l+1$, then
\beqs ma\leq(5l+1)a<\frac{3}{2}\times(5l+1)\times\left(\frac{(l-1)n}{5l-6}-\frac{(l-1)n}{5l-3}\right)=\frac{9(l-1)(5l+1)n}{2(5l-6)(5l-3)}<n,\eeqs
and we have done.

Next assume that $l\in[3,9]$. If $\gcd(5l-4,n)=1$, let $m=5l-4$ and $k=l-1$. Then
\beqs ma\leq(5l-4)a<\frac{3}{2}\times(5l-4)\times\left(\frac{(l-1)n}{5l-6}-\frac{(l-1)n}{5l-3}\right)=\frac{9(l-1)(5l-4)n}{2(5l-6)(5l-3)}\leq n,\eeqs
as desired. Hence we may assume that $\gcd(5l-4,n)>1$. This implies that $\gcd(5l-1,n)=1$. Now let $m=5l-1$ and $k=l$, by Lemma 3.3, $\ind(S)=1$.

{\it Subcase 3.3.} (iii) holds. This subcase can be proved in a similar way to Subcase 3.2.
\end{proof}

\begin{lemma}
If $6<\frac{n}{c}\leq7<\frac{n}{b}<8$ and $7|n$, then $\ind(S)=1$.
\end{lemma}
\begin{proof}
Since $6<\frac{n}{c}\leq7<\frac{n}{b}<8$, we have $\frac{n}{8}<b<\frac{n}{7}\leq c<\frac{n}{6}$. Note that $m_1=7$.

If $l=2$, then $12 < \frac{2n}{c}\leq 13 < 14 < 15 < \frac{2n}{b}<16$. If $\gcd(15,n) = 1$, let $m = 15$ and $k = 2$, otherwise let $m = 13$ and
$k = 2$. Then
\beqs ma\leq15a\leq15\times{3}{2}(c-b)<\frac{45}{2}\times\left(\frac{n}{6}-\frac{n}{8}\right)<n,\eeqs
and we have done.

Next assume that $l\geq3$. Recall that $7l-3 < \frac{ln}{c}\leq7l <\frac{ln}{b} < 7l+ 3$. We distinguish two
cases according to the number of integers contained in $\left[\frac{ln}{c},\frac{ln}{b}\right)$.

{\bf Case 1.}  There exist exactly three integers in $\left[\frac{ln}{c},\frac{ln}{b}\right)$.

Then $7l-t <\frac{ln}{c}\leq 7l-t+1< 7l-t+2 <7l-t+3 <\frac{ln}{b}\leq 7l-t+4$ for some $t\in[1,3]$. Let $k=l$ and $m\in[7l-t+1,7l-t+3]$ such that $\gcd(n,m)=1$.  Then
\beqs &&ma\leq(7l -t + 3)a \leq\frac{3(7l-t+2)}{2}(c-b)\\
&&<\frac{3(7l-t+3)}{2}\left(\frac{ln}{7l-t}-\frac{ln}{7l-t+4}\right)=\frac{(7l-t+3)\times6ln}{(7l-t)(7l-t+4)}<n,\eeqs
and we have done.

{\bf Case 2.}  There exist exactly four integers in $\left[\frac{ln}{c},\frac{ln}{b}\right)$.

First we have $7l-2 <\frac{ln}{c}\leq 7l-1< 7l <7l+1<7l+2 <\frac{ln}{b}\leq 7l+3$ or $7l-3 <\frac{ln}{c}\leq 7l-2< 7l-1 <7l<7l+1 <\frac{ln}{b}\leq 7l+2$. Then there exists $m\leq7l+1$ contained in $\left[\frac{ln}{c},\frac{ln}{b}\right)$ such that $\gcd(n,m)=1$.

By the minimality of $l$, we have
\beqs 7(l-1)-1<\frac{(l-1)n}{c}\leq 7(l-1)<7(l-1)+1<\frac{(l-1)n}{b}\leq 7(l-1)+2,\eeqs
or
\beqs 7(l-1)-2<\frac{(l-1)n}{c}\leq 7(l-1)-1<7(l-1)<\frac{(l-1)n}{b}\leq 7(l-1)+1.\eeqs
Then
$$ma\leq(7l-1)a<\frac{3(7l-1)}{2}\times\left(\frac{(l-1)n}{7l-8}-\frac{(l-1)n}{7l-5}\right)<n,$$
or
$$ma\leq(7l-1)a<\frac{3(7l-1)}{2}\times\left(\frac{(l-1)n}{7l-9}-\frac{(l-1)n}{7l-6}\right)\leq n,$$
and we have done.
\end{proof}

Now we are in a position to prove Proposition 2.6.

{\noindent\it Proof of Proposition 2.6.}

Recall that either $m_1 = 5$ or $m_1 = 7$ or $m_1\geq10$. By Lemmas 3.5 and 3.6 we may assume
$m_1\geq10$. Then $n\geq m_1b\geq10b$. Let $k = l$ and let $m$ be one of the integers in $\left[\frac{ln}{c}, \frac{ln}{b}\right)$ which is co-prime to $n$. Recall that we have either (3.3) holds or (3.4) holds.

If (3.2) holds, then $(lm_1-2)(b+a-e) = (lm_1-2)c <ln\leq(lm_1+3)b$, so $(lm_1-2)(a-e) < 5b$. Note that $m\leq lm_1+2$ and $l\geq 2$, then
\beqs ma\leq (lm_1 + 2)a =\frac{lm_1 + 2}{lm_1-2}\times\frac{a}{a-e}\times(lm_1-2)(a-e)<\frac{2\times10 + 2}{2\times10-2}\times\frac{3}{2}\times5b< 10b\leq n,\eeqs
and we are done.

If (3.3) holds, then $(lm_1-3)(b+a-e) = (lm_1-3)c <ln\leq(lm_1+2)b$, so $(lm_1-3)(a-e) < 5b$. Note that $m\leq lm_1+1$ and $l\geq 2$, then
\beqs ma\leq (lm_1 + 1)a =\frac{lm_1 + 1}{lm_1-3}\times\frac{a}{a-e}\times(lm_1-3)(a-e)<\frac{2\times10 + 1}{2\times10-3}\times\frac{3}{2}\times5b< 10b\leq n,\eeqs
and we are done.

\section{Proof of Proposition 2.7}

In this section, we always assume that $\lceil\frac{n}{c}\rceil=\lceil\frac{n}{b}\rceil$, so $k_1\geq2$, and we also assume that $k_1>\frac{b}{a}$. Proposition 2.7 can be proved through the following three lemmas.

\begin{lemma}If the assumption is as in Proposition 2.7, then $k_1<4$.\end{lemma}
\begin{proof}If $k_1\geq4$, then $\frac{(k_1-1)n}{b}-\frac{(k_1-1)n}{c}=\frac{(a-e)(k_1-1)n}{bc}\geq\frac{2a}{3}\frac{3k_1n}{4bc}>1$, a contradiction. \end{proof}

\begin{lemma}If the assumption is as in Proposition 2.7, then $k_1\not=3$.\end{lemma}
\begin{proof}If $a\geq4e$, then $\frac{(k_1-1)n}{b}-\frac{(k_1-1)n}{c}=\frac{(a-e)2n}{bc}\geq\frac{3a}{4}\frac{2n}{bc}>1$, a contradiction. Hence we assume that $3e<a<4e$, and $e<\frac{3p^{i_0}}{2}$.

If $\frac{n}{c}>\frac{9}{4}$, then $\frac{(k_1-1)n}{b}-\frac{(k_1-1)n}{c}=\frac{(a-e)2n}{bc}\geq\frac{2a}{3}\frac{2n}{bc}>1$, a contradiction.

If $\frac{n}{c}<\frac{9}{4}<\frac{n}{b}<\frac{5}{2}$, then $9a>3b>\frac{6n}{5}$, and $n<\frac{45a}{6}<\frac{45\times4\times\frac{3}{2}p^{i_0}}{6}=45p^{i_0}$, a contradiction.

If $\frac{n}{c}<\frac{n}{b}<\frac{9}{4}$, then $9a>3b>\frac{4n}{3}$, $n<27e<\frac{81}{2}p^{i_0}$, a contradiction.
\end{proof}

\begin{lemma}If the assumption is as in Proposition 2.7 and $k_1=2$, then $\ind(S)=1$.\end{lemma}
\begin{proof}
If $\frac{n}{c}>3$, then $\frac{n}{b}-\frac{n}{c}=\frac{(a-e)n}{bc}\geq\frac{2a}{3}\frac{n}{bc}>1$, a contradiction.

If $\frac{n}{c}\leq3<\frac{n}{b}$, we have $n<3c<2n$, $3a<3b<n$. Let $m=3$, then $\gcd(n,m)=1$ and
$|me|_n+|mc|_n+|m(n-b)|_n+|m(n-a)|_n=me+(mc-n)+(n-mb)+(n-ma)=n$, we have done.

If $\frac{n}{c}<\frac{n}{b}<3$, then $\frac{n}{3}<b<2a$, and $2n<6c<3n$, $2n<6b<3n$, $6a>3b>n$. $6e<2a<n$. Let $m=6$, then $\gcd(n,m)=1$, and $3n\geq|me|_n+|mc|_n+|m(n-b)|_n+|m(n-a)|_n\geq me+(mc-2n)+(3n-mb)+(2n-ma)=3n$, we have done.
\end{proof}

\section{Proof of Proposition 2.8}

In this section, we always assume that $\lceil\frac{n}{c}\rceil=\lceil\frac{n}{b}\rceil$, so $k_1\geq2$, and we also assume that $k_1<\frac{b}{a}$, hence $s\geq k_1$.

\begin{lemma}If the assumption is as in Proposition 2.8, then $k_1\not=7$.\end{lemma}
\begin{proof}If $k_1=7$, then $s=7$, and $\frac{(k_1-1)n}{b}-\frac{(k_1-1)n}{c}=\frac{(a-e)6n}{bc}\geq\frac{2\times 8a}{3b}\frac{3}{4}\frac{n}{c}>1$, a contradiction.\end{proof}

\begin{lemma}If the assumption is as in Proposition 2.8 and $k_1=6$, then $\ind(S)=1$.\end{lemma}
\begin{proof}If $k_1=6$, we have $\frac{n}{c}<\frac{12}{5}$, otherwise $\frac{(k_1-1)n}{b}-\frac{(k_1-1)n}{c}=\frac{(a-e)5n}{bc}\geq\frac{2\times 8a}{3b}\frac{5}{8}\frac{n}{c}>\frac{10n}{24c}\geq1$, a contradiction.
So we have $10<\frac{5n}{c}<\frac{5n}{b}\leq 11$ or $11<\frac{5n}{c}<\frac{5n}{b}\leq 12$.

{\bf Case 1.} $10<\frac{5n}{c}<\frac{5n}{b}\leq11$.

It holds that $12<\frac{6n}{c}\leq13<\frac{6n}{b}\leq\frac{66}{5}$ and $16<\frac{8n}{c}\leq 17<\frac{8n}{b}\leq\frac{88}{5}$.

If $17a\geq n$, then $8n<18b<18c<9n$ and $18e<6a<b<n$. Let $m=18$, then $\gcd(n,m)=1$ and
$|me|_n+|mc|_n+|m(n-b)|_n+|m(n-a)|_n\geq 18e+(18c-8n)+(9n-18b)+(2n-18a)=3n$, hence $\ind(S)=1$.

Assume that $17a<n$, then at least one of $\{13,17\}$ co-prime to $n$ through Lemma 2.4(iv), which says $5|n$.
Then we have done.

{\bf Case 2.} $11<\frac{5n}{c}<\frac{5n}{b}\leq12$.

It holds that $\frac{77}{5}<\frac{7n}{c}<16<\frac{7n}{b}\leq\frac{84}{5}$. Since $a\leq\frac{3(a-e)}{2}=\frac{3(c-b)}{2}<\frac{3}{2}\times\left(\frac{5n}{11}-\frac{5n}{12}\right)=\frac{5n}{88}<\frac{n}{17}$, we have $16a<17a<n$. Let $m=16$, then $\gcd(n,m)=1$ and
$|me|_n+|mc|_n+|m(n-b)|_n+|m(n-a)|_n\geq 16e+(16c-7n)+(7n-16b)+(n-16a)=n$, hence $\ind(S)=1$.
\end{proof}

\begin{lemma}
If the assumption is as in Proposition 2.8 and $k_1=5$, then $\ind(S)=1$.
\end{lemma}
\begin{proof}If $k_1=5$, we have $\frac{n}{c}<3$, otherwise $\frac{(k_1-1)n}{b}-\frac{(k_1-1)n}{c}=\frac{(a-e)4n}{bc}\geq\frac{2\times 4a}{3b}\frac{n}{c}>\frac{n}{c}\geq1$, a contradiction.
So it holds $8+t<\frac{4n}{c}<\frac{4n}{b}\leq 9+t$ for some $t=0,1,2,3$.

{\bf Case 1.} $t=0$. We have $a\leq\frac{3(a-e)}{2}=\frac{3(c-b)}{2}<\frac{3}{2}\times\left(\frac{n}{2}-\frac{4n}{9}\right)=\frac{n}{12}$.

If $\gcd(n,11)=1$,  we have $10<\frac{5n}{c}\leq 11<\frac{5n}{b}<\frac{45}{4}$. Let $m=11$, then $\ind(S)=1$.

If $15a>n$, we have $14<\frac{7n}{c}\leq 15<\frac{7n}{b}<\frac{63}{4}<16$ and $7n<16b<16c<8n$ and $16e<6a<n$. Let $m=16$, then $|me|_n+|mc|_n+|m(n-b)|_n+|m(n-a)|_n\geq 16e+(16c-7n)+(8n-16b)+(2n-16a)=3n$, hence $\ind(S)=1$.

If $15a<n$ and $\gcd(n,5)=1$, let $m=15$, we have $\ind(S)=1$.

If $15a\leq n$ and $5|n, 11|n$, we have $12<\frac{6n}{c}\leq 13<\frac{6n}{b}<\frac{27}{2}$. Let $m=13$, we have $\ind(S)=1$.

{\bf Case 2.} $t=1$. We have $a\leq\frac{3(a-e)}{2}=\frac{3(c-b)}{2}<\frac{3}{2}\times\left(\frac{4n}{9}-\frac{4n}{10}\right)=\frac{n}{15}$.
Since $\frac{45}{4}<\frac{5n}{c}<12<\frac{5n}{b}<\frac{50}{4}$, let $m=12$, then $\ind(S)=1$.

{\bf Case 3.} $t=2$. We have $a\leq\frac{3(a-e)}{2}=\frac{3(c-b)}{2}<\frac{3}{2}\times\left(\frac{4n}{10}-\frac{4n}{11}\right)=\frac{3n}{55}<\frac{n}{18}$.
Since $15=\frac{60}{4}<\frac{6n}{c}<16<\frac{6n}{b}<\frac{66}{4}<17$, let $m=16$, then $\ind(S)=1$.

{\bf Case 4.} $t=3$. We have $a\leq\frac{3(a-e)}{2}=\frac{3(c-b)}{2}<\frac{3}{2}\times\left(\frac{4n}{11}-\frac{4n}{12}\right)=\frac{n}{22}$.
We have \beqs \frac{55}{4}<\frac{5n}{c}<14<\frac{5n}{b}<15,\\
\frac{66}{4}<\frac{6n}{c}<17<\frac{6n}{b}<18,\\
\frac{77}{4}<\frac{7n}{c}<20<\frac{7n}{b}<21,\eeqs.
At least one of $\{14,17,20\}$ coprime to $n$. Let $m$ be one of $\{14,17,20\}$ such that $\gcd(n,m)=1$, then $|me|_n+|mc|_n+|m(n-b)|_n+|m(n-a)|_n=n$ and $\ind(S)=1$.
\end{proof}

\begin{lemma}If the assumption is as in Proposition 2.8 and  $k_1=4$, then $\ind(S)=1$.\end{lemma}
\begin{proof}If $k_1=4$, we have $s\geq4$ and $\frac{n}{b}<4$.
So it holds $6+t<\frac{3n}{c}<\frac{3n}{b}\leq 7+t$ for some $t=0,1,2,3,4,5$.

{\bf Case 1.} $t=0$. We have $a\leq\frac{3(a-e)}{2}=\frac{3(c-b)}{2}<\frac{3}{2}\times\left(\frac{n}{2}-\frac{3n}{7}\right)=\frac{3n}{28}<\frac{n}{9}$,
and $8<\frac{4n}{c}<9<\frac{4n}{b}<\frac{28}{3}$. Let $m=9$, then $\ind(S)=1$.

{\bf Case 2.} $t=1$. We have $a\leq\frac{3(a-e)}{2}=\frac{3(c-b)}{2}<\frac{3}{2}\times\left(\frac{3n}{7}-\frac{3n}{8}\right)=\frac{9n}{112}<\frac{n}{12}$.
If $\frac{35}{3}<\frac{5n}{c}<12<\frac{5n}{b}<\frac{40}{3}$, let $m=12$, then $\ind(S)=1$. If $12<\frac{5n}{c}\leq13<\frac{5n}{b}<\frac{40}{3}$, we have $a<\frac{3}{2}\times\left(\frac{5n}{12}-\frac{3n}{8}\right)=\frac{n}{16}$, hence $\ind(S)=1$ in case of $\gcd(n,13)=1$. We also have $\ind(S)=1$ in case of $\gcd(n,13)=1$ since $\frac{28}{3}<\frac{4n}{c}<10<\frac{4n}{b}<\frac{32}{3}$.

Assume that $5|n, 13|n$ and $12<\frac{5n}{c}\leq13<\frac{5n}{b}<\frac{40}{3}$. Hence we have $\frac{84}{5}<\frac{7n}{c}<\frac{7n}{b}<\frac{56}{3}$.

If $18a>n$, let $m=19$. Then $me=19<n$, $7n<mb<mc<\frac{96c}{5}=\frac{8}{7}\times\frac{84c}{5}<8n$. Hence we have
$\gcd(n,m)=1$ and
$|me|_n+|mc|_n+|m(n-b)|_n+|m(n-a)|_n\geq 19e+(19c-7n)+(8n-19b)+(2n-19a)=3n$. So $\ind(S)=1$.

If $18a<n$, there exists $m\in\{17,18\}$ such that $\frac{7n}{c}\leq m<\frac{7n}{b}$, $ma<n$ and $\gcd(n,m)=1$,  then we have done.

{\bf Case 3.} $t=2$. We have $a\leq\frac{3(a-e)}{2}=\frac{3(c-b)}{2}<\frac{3}{2}\times\left(\frac{3n}{8}-\frac{3n}{9}\right)=\frac{n}{16}$.

If $\gcd(n,11)=1$ or $\gcd(n,7)=1$, by inequalities
$\frac{32}{3}<\frac{4n}{c}\leq 11<\frac{4n}{b}<12, \frac{40}{3}<\frac{5n}{c}\leq 14<\frac{5n}{b}<15$, it is easy to show that $\ind(S)=1$.

Assume that $11|n, 7|n$. We have $16<\frac{6n}{c}\leq 17<\frac{6n}{b}<18$.

If $17a<n$, let $m=17$, we have done.

If $17a>n$, let $m=18$. Then $6n<mb<mc=\frac{9}{8}16c<\frac{7}{6}16c<7n$, and

$|me|_n+|mc|_n+|m(n-b)|_n+|m(n-a)|_n\geq 18e+(18c-6n)+(7n-18b)+(2n-18a)=3n$. So $\ind(S)=1$.

{\bf Case 4.} $t=3$. We have $a\leq\frac{3(a-e)}{2}=\frac{3(c-b)}{2}<\frac{3}{2}\times\left(\frac{3n}{9}-\frac{3n}{10}\right)=\frac{n}{20}$,
and $15<\frac{5n}{c}<16<\frac{5n}{b}<\frac{50}{3}<17$. Let $m=16$, then $\ind(S)=1$.

{\bf Case 5.} $t=4$. We have $a\leq\frac{3(a-e)}{2}=\frac{3(c-b)}{2}<\frac{3}{2}\times\left(\frac{3n}{10}-\frac{3n}{11}\right)<\frac{n}{24}$ and $\frac{50}{3}<\frac{5n}{c}<16<\frac{5n}{b}<\frac{55}{3}$.

If $\frac{50}{3}<\frac{5n}{c}<18<\frac{5n}{b}<\frac{55}{3}$, let $m=18$. Then $\ind(S)=1$.

If $\frac{50}{3}<\frac{5n}{c}\leq17<\frac{5n}{b}<18$, we have $30a<n$. Then $n>30a>\frac{15b}{4}>\frac{15}{4}\times\frac{5n}{18}=\frac{25n}{24}>n$, it is a contradiction.

{\bf Case 6.} $t=5$. We have $a\leq\frac{3(a-e)}{2}=\frac{3(c-b)}{2}<\frac{3}{2}\times\left(\frac{3n}{11}-\frac{3n}{12}\right)<\frac{n}{29}$ and $\frac{11}{3}<\frac{n}{c}<\frac{n}{b}<4$.
So \beqs &\frac{44}{3}<\frac{4n}{c}\leq15<\frac{4n}{b}<16,\\
&\frac{55}{3}<\frac{5n}{c}\leq19<\frac{5n}{b}<20,\\
&22<\frac{n}{c}\leq23<\frac{n}{b}<24.\eeqs
Then there exists at least one of integers $15,19,23$ coprime to $n$. So it is clear that $\ind(S)=1$.
\end{proof}

\begin{lemma}If the assumption is as in Proposition 2.8 and $k_1=3$, then $\ind(S)=1$.\end{lemma}
\begin{proof}If $k_1=3$, we have $\frac{n}{b}<6$.
So it holds $4+t<\frac{2n}{c}<\frac{2n}{b}\leq 5+t$ for some integer $t\in[0,7]$.

{\bf Case 1.} $t=0$. $6<\frac{3n}{c}\leq7<\frac{3n}{b}\leq\frac{15}{2}$.

If $8a>n$, let $m=8$. Then $3n<8b<8c<4n$, $8e<3a<b<n$ and $|me|_n+|mc|_n+|m(n-b)|_n+|m(n-a)|_n\geq 8e+(8c-3n)+(4n-8b)+(2n-8a)=3n$. So $\ind(S)=1$.

If $8a<n$, since $8<\frac{4n}{c}<9<\frac{4n}{b}\leq10$, let $m=9$. Then $\ind(S)=1$.

{\bf Case 2.} $t=1$. We have $\frac{15}{2}<\frac{3n}{c}<8<\frac{3n}{b}<9$ and $a\leq\frac{3(a-e)}{2}=\frac{3(c-b)}{2}<\frac{3}{2}\times\left(\frac{2n}{5}-\frac{n}{3}\right)=\frac{n}{10}$.
Let $m=8$, then $\gcd(n,m)=1$ and $|me|_n+|mc|_n+|m(n-b)|_n+|m(n-a)|_n=n$, hence $\ind(S)=1$.

{\bf Case 3.} $t=2$. We have $9<\frac{3n}{c}<10<\frac{3n}{b}<\frac{21}{2}$ and $a\leq\frac{3(a-e)}{2}=\frac{3(c-b)}{2}<\frac{3}{2}\times\left(\frac{n}{3}-\frac{2n}{7}\right)=\frac{n}{14}$.

If $17a\geq n$, let $m=18$, then $5n<18b<18c=\frac{6}{5}\times15c<6n$ and $18e<6a<n$, we have
$|me|_n+|mc|_n+|m(n-b)|_n+|m(n-a)|_n\geq 3n$, hence $\ind(S)=1$.

If $17a<n$ and $15<\frac{5n}{c}<16<\frac{5n}{b}\leq \frac{35}{2}$, let $m=16$. Then $\ind(S)=1$.

Assume that $16<\frac{5n}{c}\leq 17<\frac{5n}{b}\leq \frac{35}{2}$, then $a\leq\frac{3(a-e)}{2}=\frac{3(c-b)}{2}<\frac{n}{24}$.
We also have $9<\frac{3n}{c}\leq <10<\frac{3n}{b}\leq \frac{21}{2}$ and $12<\frac{4n}{c}<13<\frac{4n}{b}<14$. Then at least one of integers $10, 13, 17$ is co-prime to $n$, and we have done.

{\bf Case 4.} $t=3$. We have $\frac{7}2<\frac{n}{c}<\frac{n}{b}<4$ and $a\leq\frac{3(a-e)}{2}=\frac{3(c-b)}{2}<\frac{3}{2}\times\left(\frac{2n}{7}-\frac{n}{4}\right)<\frac{n}{18}$.

At first we have $\frac{35}{2}<\frac{5n}{c}<\frac{5n}{b}<20$.

If  $\frac{5n}{c}<18<\frac{5n}{b}$, let $m=18$, then we have done.

If $18<\frac{5n}{c}\leq 19<\frac{5n}{b}<20$, we have $a<\frac{n}{24}$. Since $\frac{21}2<\frac{3n}{c}\leq 11<\frac{3n}{b}<12$, $14<\frac{4n}{c}<15<\frac{4n}{b}<16$ and at least one of integers $11,15,19$ is co-prime to $n$, then it is easy to show that $\ind(S)=1$.

{\bf Case 5.} $t=4$. We have $4<\frac{n}{c}<\frac{n}{b}<\frac{9}{2}$ and $a\leq\frac{3(a-e)}{2}=\frac{3(c-b)}{2}<\frac{3}{2}\times\left(\frac{n}{4}-\frac{2n}{9}\right)=\frac{n}{24}$.

We also have
\beqs &12<\frac{3n}{c}\leq13<\frac{3n}{b}<\frac{27}{2},\\
&16<\frac{4n}{c}\leq17<\frac{4n}{b}<18,\\
&20<\frac{5n}{c}\leq m_1<\frac{5n}{b}<\frac{45}{2},\eeqs
where $m\in\{21,22\}$. It is easy to see that at least one of integers $13,17, m_1$ is co-prime to $n$. Then $\ind(S)=1$.

{\bf Case 6.} $t=5$. We have $\frac{9}2<\frac{n}{c}<\frac{n}{b}\leq5$ and $a\leq\frac{3(a-e)}{2}=\frac{3(c-b)}{2}<\frac{3}{2}\times\left(\frac{2n}{9}-\frac{n}{5}\right)=\frac{n}{30}$.

If $\frac{5n}{c}<24<\frac{5n}{b}\leq25$, then let $m=24$ and we have done.
Otherwise, we have \beqs &\frac{27}{2}<\frac{3n}{c}\leq14<\frac{3n}{b}\leq15,\\
&18<\frac{4n}{c}\leq19<\frac{4n}{b}\leq20,\\
&\frac{45}{2}<\frac{5n}{c}\leq23<\frac{5n}{b}<24,\eeqs
there exists at least one of integers $14,19, 23$ is co-prime to $n$. Then $\ind(S)=1$.

{\bf Case 7.} $t=6$. We have $5<\frac{n}{c}<\frac{n}{b}\leq\frac{11}{2}$ and $a\leq\frac{3(a-e)}{2}=\frac{3(c-b)}{2}<\frac{3}{2}\times\left(\frac{n}{5}-\frac{2n}{11}\right)<\frac{n}{36}$.

We also have $15<\frac{3n}{c}<16<\frac{3n}{b}\leq\frac{33}{2}$, let $m=16$. Then $ma<n$ and $\ind(S)=1$.

{\bf Case 8.} $t=7$. We have $\frac{11}{2}<\frac{n}{c}<\frac{n}{b}<6$ and $a\leq\frac{3(a-e)}{2}=\frac{3(c-b)}{2}<\frac{3}{2}\times\left(\frac{2n}{11}-\frac{n}{6}\right)=\frac{n}{44}$.

We also have \beqs &\frac{33}{2}<\frac{3n}{c}\leq17<\frac{3n}{b}<18,\\
&22<\frac{4n}{c}\leq23<\frac{4n}{b}<24,\\
&\frac{55}{2}<\frac{5n}{c}\leq m_1<\frac{5n}{b}<30,\eeqs
where $m_1\in\{28,29\}$, and there exists at least one of integers $17,23, m_1$ is co-prime to $n$. Then $\ind(S)=1$.
\end{proof}

\begin{lemma}
Let $e,a,b,c$ be parameters listed in Proposition 2.5. If $n=5^\alpha7^\beta$ and $\frac{3n}{8}<b<c<\frac{11n}{23}$, then $\frac{n}{9}\geq a$.
\end{lemma}
\begin{proof}
{\bf case 1.} $e=p^{i_0}$:

If $e=5$ or $e=7$, then $n>\frac{1000}{7}e\geq142e$. If $e\geq25$, then $n\geq 5p^{i_0}q^{j_0}\geq 5e^2\geq 125e$.

{\bf case 2.} $e=q^{j_0}$:

If $e=7$, $n>142e$. Clearly, $e$ can't equal to $25$, otherwise we can't find suitable $p^{i_0}$. When $e=49$, we have $p^{i_0}=25$ and $n\geq 5p^{i_0}e=125e$. If $e\geq125$, we have $p^{i_0}>\frac{e}{3}$ and $n\geq 5p^{i_0}>208$.

Both of the above cases, we have $n\geq 125e$.
If $\frac{n}{9}< a$, then
\beqs \frac{n}{9}<a<\frac{11n-p^{i_0}}{23}-\frac{3n+q^{j_0}}{8}+e\leq\frac{19n+169e}{184},\eeqs
hence we have $n<117e$, which contradicts to $n\geq125e$.

{\bf case 3.} $e=2q^{j_0}$. Clearly, $e\not\in\{10,50\}$.

{\it subcase 3.1}. $e=14$. If $n\geq 5^47$, then $n\geq\frac{5^4}{2}e>322e$. The proof is similar to above.

Otherwise $n=5^27^2$. Then $a\in\{2\times(2t+1)\times7, n-\frac{n}{7}+10\}$. Since $5|(2t+1-1)$, we have $t\geq5$.
Moreover, $n-\frac{n}{7}+10=75\times14+10$. So $a\geq11e$.
Then we have
\beqs a\leq\frac{11}{10}(a-e)<\frac{11}{10}\left(\frac{11n}{23}-\frac{3n}{8}\right)=\frac{201n}{1840}=\frac{n}{9}\times\frac{1809}{1840}<\frac{n}{9}.\eeqs

{\it subcase 3.2}. $e=98$. The proof is similar to subcase 3.1.

{\it subcase 3.3}. $e\geq250$, we have $n>312e$ and the proof is similar to Case 1 and Case 2.
\end{proof}

\begin{lemma}
Let $k_1=2$, $4<\frac{2n}{c}\leq 5<\frac{2n}{b}<6$ and $a\leq \frac{b}{2}$. If the assumption is as in Proposition 2.8, then $\ind(S)=1$.
\end{lemma}

\begin{proof}
Then $4<\frac{2n}{c}\leq5<\frac{2n}{b}<6$.
If $6a>n$, then $2n<6c,6b<3n$, $n<6a<2n$, $6e<2a<n$, we have $|me|_n+|mc|_n+|m(n-b)|_n+|m(n-a)|_n=3n$.

If $6a<n$ and $\gcd(n,5)=1$, let $m=5$, we have $\frac{2n}{c}\leq5<\frac{2n}{b}$, then  $|me|_n+|mc|_n+|m(n-b)|_n+|m(n-a)|_n=n$.

Next we assume that $5|n$ and $6a<n$.

{\bf Case 1.} $7<\frac{3n}{c}<8<\frac{3n}{b}<9$.
If $8a<n$, let $m=8$, we have done.

If $8a>n$, let $m=9$. Then $3n<9b<9c<\frac{27n}{7}<4n$ and $9e<3a<n$. We have $|me|_n+|mc|_n+|m(n-b)|_n+|m(n-a)|_n\geq3n$, hence $\ind(S)=1$.

{\bf Case 2.} $6<\frac{3n}{c}<7<\frac{3n}{b}<8$ and $\gcd(n,7)=1$. We have $a<\frac{3}{2}\left(\frac{n}{2}-\frac{3n}{8}\right)<\frac{n}{5}$.

If $7a<n$, let $m=7$, we have done.

If $7a>n$, let $m=14$. Then $6n<14c<7n$, $5n<\frac{40b}{3}<14b<6n$ and $14e<5a<n$. We have $|me|_n+|mc|_n+|m(n-b)|_n+|m(n-a)|_n\geq3n$, hence $\ind(S)=1$.

{\bf Case 3.} $6<\frac{3n}{c}<7<\frac{3n}{b}<8$ and $\gcd(n,7)>1$.

Note that $8<\frac{4n}{c}\leq10<\frac{4n}{b}<12$.

If $9<\frac{4n}{c}\leq10<\frac{4n}{b}<12$, we have
$\frac{5n}{c}\leq\frac{35}{3}<12=\frac{10\times6}{5}<\frac{5n}{b}$
and
\beqs a&<&\left\{\begin{array}{l}\left(\frac{4n}{9}-\frac{3n}{8}+\frac{n}{75}\right)<\frac{n}{12},\quad e=p^{i_0},\\
\frac{6}{5}\times\left(\frac{4n}{9}-\frac{3n}{8}\right)=\frac{n}{12},\quad e\not=p^{i_0},
\end{array}\right.\eeqs
let $m=12$ and $k=5$, then we have done.

If $8<\frac{4n}{c}<9<10<\frac{4n}{b}$, then $\frac{3n}{8}<b<\frac{2n}{5}<\frac{4n}{9}<c$ and
\beqs 8n+\frac{n}{2}<\frac{69n}{8}<23b<\frac{46n}{5}<9n+\frac{n}{2}<10n<\frac{92n}{9}<23c<\frac{23n}{2}=11n+\frac{n}{2}.\eeqs

Note that $a=c-b+e\leq\frac{n-p^{i_0}}{2}-\frac{3n+p^{i_0}}{8}+e=\frac{n-5p^{i_0}}{8}+e$. If $a>\frac{n}{8}$, then let $M=12$. We obtain that $|Me|_n<\frac{n}{2}$, $|Mb|_n>\frac{n}{2}$ and $|Ma|_n>\frac{n}{2}$ since
\beqs \frac{3n}{2}<Ma\leq \frac{3n}{2}+12e-\frac{15p^{i_0}}{2}\eeqs
and
\beqs
12e-\frac{15p^{i_0}}{2}&\leq&\left\{\begin{array}{l}9p^{i_0}<\frac{3n}{25}<\frac{n}{2},\quad e=p^{i_0},\\
12e\leq 2a<\frac{n}{2}, \quad e\not=p^{i_0},
\end{array}\right.\eeqs
and we have done.

If $9a<n$, let $m=9,k=4$. Then $\ind(S)=1$.

Then we assume that $\frac{n}{9}<a<\frac{n}{8}$, and thus
\beqs 9n=\frac{3n}{8}\times24<24b<24\times\frac{2n}{5}<10n<24\times\frac{4n}{9}<24c<12n.\eeqs
By Lemma 5.6, we have $23c>11n$. Then $|23c|_n<\frac{n}{2}$. By Proposition 2.5, we have $|23e|_n=23e<\frac{n}{2}$.
We also have $\frac{5n}{2}<\frac{23n}{9}<23a<\frac{23n}{8}<3n$, hence $|23a|_n>\frac{n}{2}$. Then we have $\ind(S)=1$.

{\bf Case 4.} $6<\frac{3n}{c}\leq7<8<\frac{3n}{b}<9$. We distinguish three subcases.

{\it Subcase 4.1.} $\gcd(n,77)=1$.

We may assume that $a >\frac{n}7$ (for otherwise, if let $m = 7$ and $k = 3$, we have $ma < n$, so the
lemma follows from Lemma 2.3 (1)). Hence $n < 11a < 2n$. Also, we have that
$3n < \frac{11n}{3}<11b < \frac{33n}{8} < 5n$
and
$4n < \frac{33n}7 < 11c < \frac{11n}2 < 6n$.

If $11b < 4n$ and $11c > 5n$, we have $|11e|_n + |11c|_n + |11(n-b)|_n + |11(n-a)|_n = 11e + (11c-5n) + (4n-11b) + (2n -11a) = n$ and thus $\ind(S) = 1$.

If $11b > 4n$ and $11c < 5n$, we have $|11e|_n + |11c|_n + |11(n-b)|_n + |11(n-a)|_n = 11e + (11c-4n) + (5n-11b) + (2n -11a) = 3n$ and thus $\ind(S) = 1$ (by Remark 2.1 (2)).

If $11b < 4n$ and $11c < 5n$, then we have either $\frac{n}7 < a = c-b + e\leq\frac{5n}{11}-\frac{n}3+e$, which implies that $n < 47e$, or
$\frac{n}7 < a\leq\frac{25}{24}(a-e)=\frac{25}{24}(c-b)<\frac{25n}{198}<\frac{25n}{175}=\frac{n}{7}$. By Lemma 3.1, both of them lead to a contradiction.

If $11b > 4n$ and $11c > 5n$, then
either $\frac{n}7 < a = c-b + e\leq\frac{n-e}{2}-\frac{4n-e}{11}+e$, which implies that $n < 63e$, or
$\frac{n}7 < a\leq\frac{25}{24}(a-e)=\frac{25}{24}(c-b)<\frac{25n}{176}<\frac{25n}{175}=\frac{n}{7}$. By Lemma 3.1, both of them lead to a contradiction.

{\it Subcase 4.2.} $55|n$.

As in Subcase 4.1, we may assume that $a > \frac{n}7$. Then
\beqs\frac{3n}2 < \frac{13n}{7} < 13a < \frac{13n}{6}<\frac{5n}2 < 4n <\frac{13n}3 < 13b <\frac{39n}8 < 5n < \frac{11n}2 <\frac{39n}7 < 13c < \frac{13n}2.\eeqs

If $13c < 6n$, then $\frac{n}{7} < a = c-b + e\leq\frac{6n}{13}-\frac{n}3+e$, so $n < 69e$, yielding
a contradiction by Lemma 3.1. Hence we must have that $13c > 6n$, and then $|13c|_n <\frac{n}2$. Moreover, we have $13e<\frac{n}{2}$ by Lemma 3.1.

If $13a < 2n$ or $13b > \frac{9n}2$, then $|13a|_n>\frac{n}2$ or $|13b|_n > \frac{n}2$. Since $\gcd(n, 13) = 1$, the lemma follows from Lemma 2.3 (2) with $M = 13$. Next we assume that $13a > 2n$ and $13b < \frac{9n}2$. Then $\frac{2n}{13} < a < \frac{n}{6}$ and $\frac{n}3 < b < \frac{9n}{26}$. Therefore,
\beqs \frac{5n}2<\frac{34n}{13}<17a <\frac{17n}6< 3n <\frac{11n}2<\frac{17n}3< 17b <\frac{153n}{26}< 6n.\eeqs
We infer that $|17a|_n >\frac{n}2$ and $|17b|_n >\frac{n}2$. Since $\gcd(n, 17) = 1$ and $17e<\frac{n}{2}$, the lemma follows from Lemma 2.3 (2) with $M = 17$.

{\it Subcase 4.3.} $35|n$.
As in Subcase 4.1, we may assume that $a > \frac{n}8$. By using a similar argument in Subcase 4.2 and Lemma 3.1,
we can complete the proof with $M = 11$ or $M = 13$.
\end{proof}

\begin{lemma}If the assumption is as in Proposition 2.8 and $k_1=2$, then $\ind(S)=1$.\end{lemma}
\begin{proof}
{\bf Case 1.} $5<\frac{n}{c}<\frac{n}{b}<6$. Then $10<\frac{2n}{c}<11<\frac{2n}{b}<12$. If $\gcd(n,11)=1$, then $a<\frac{3}{2}(a-e)=\frac{3}{2}(c-b)<\frac{3}{2}(\frac{n}{5}-\frac{n}{6})=\frac{n}{20}$, $11a<n$. Let $m=11$, $|me|_n+|mc|_n+|m(n-b)|_n+|m(n-a)|_n=n$.

Since $15<\frac{3n}{c}<\frac{33}{2}<\frac{3n}{b}<18$, if $\frac{3n}{c}<16$, then we have done. If $16<\frac{3n}{c}<17<\frac{3n}{b}<18$ and $\gcd(17,n)=1$, let $m=17$, then $|me|_n+|mc|_n+|m(n-b)|_n+|m(n-a)|_n=m$.

If $16<\frac{3n}{c}<\frac{3n}{b}<17$, then $a<\frac{3}{2}(\frac{3n}{16}-\frac{3n}{17})=\frac{3n}{272}<\frac{n}{90}<\frac{b}{15}$, a contradiction.

Now let $11|n,17|n$ and $\frac{n}{c}<\frac{11}{2}<\frac{17}{3}<\frac{n}{b}$. Then $\frac{5n}{c}<\frac{55}{2}<28<\frac{85}{3}<\frac{5n}{b}$ and $a<\frac{3}{2}(\frac{3n}{16}-\frac{n}{6})=\frac{n}{32}$, Let $m=28$, we have $\gcd(n,m)=1$ and $|me|_n+|mc|_n+|m(n-b)|_n+|m(n-a)|_n=n$.

{\bf Case 2.} $4<\frac{n}{c}<\frac{n}{b}\leq5$. Then $8<\frac{2n}{c}<9<\frac{2n}{b}\leq10$ and $a<\frac{3}{2}(a-e)=\frac{3}{2}(c-b)<\frac{3}{2}(\frac{n}{4}-\frac{n}{5})=\frac{3n}{40}$. Let $m=9$, we have $\gcd(n,m)=1$ and $|me|_n+|mc|_n+|m(n-b)|_n+|m(n-a)|_n=n$.

{\bf Case 3.} $3<\frac{n}{c}<\frac{n}{b}<4$. Then $6<\frac{2n}{c}<7<\frac{2n}{b}\leq8$ and $a<\frac{3}{2}(a-e)=\frac{3}{2}(c-b)<\frac{3}{2}(\frac{n}{3}-\frac{n}{4})=\frac{n}{8}$.
If $\gcd(n,7)=1$,  let $m=7$, we have $\gcd(n,m)=1$ and $|me|_n+|mc|_n+|m(n-b)|_n+|m(n-a)|_n=n$.

If $7|n$, we divide the proof into the following four subcases.

{\it Subcase 3.1} If $\frac{n}{c}<\frac{10}{3}<\frac{11}{3}<\frac{n}{b}$.  Then at least one of $10,11$ is co-prime to $n$. Let $m\in\{10,11\}$ be such that $\gcd(m,n)=1$. If $ma<n$, then $|me|_n+|mc|_n+|m(n-b)|_n+|m(n-a)|_n=n$.

If $ma>n$, then $3n<12c<4n$, $3n<12b<4n$,$n<12a<2n$, $12e<4a<n$, we have $|12e|_n+|12c|_n+|12(n-b)|_n+|12(n-a)|_n=3n$.

{\it Subcase 3.2} If $\frac{10}{3}<\frac{n}{c}<\frac{n}{b}<\frac{15}{4}$.  Then $a<\frac{3}{2}(\frac{3n}{10}-\frac{4n}{15})=\frac{n}{30}<\frac{b}{8}$, a contradiction.

{\it Subcase 3.3}  If $\frac{10}{3}<\frac{n}{c}<\frac{15}{4}<\frac{n}{b}$.  Then $a<\frac{3}{2}(\frac{3n}{10}-\frac{n}{4})=\frac{3n}{40}$.

 We have $\frac{4n}{c}<14<15<\frac{4n}{b}$, $\frac{6n}{c}<21<22<\frac{6n}{b}$.

 If $15a>n$, we have $4n<16c<5n$, $4n<16b<5n$, $n<16a<2n$, $16e<n$, and let $m=16$, $|me|_n+|mc|_n+|m(n-b)|_n+|m(n-a)|_n=3n$.

If $15a<n$, $\gcd(n,15)=1$, let $m=15$ we have $|me|_n+|mc|_n+|m(n-b)|_n+|m(n-a)|_n=n$.

If $22a>n$, let $m=23$, $|me|_n+|mc|_n+|m(n-b)|_n+|m(n-a)|_n=3n$.
If $22a<n$, let $m=22$, $|me|_n+|mc|_n+|m(n-b)|_n+|m(n-a)|_n=n$.

{\it Subcase 3.4.}  If $3<\frac{n}{c}<\frac{10}{3}<\frac{n}{b}<\frac{11}{3}$. Then $a<\frac{2n}{33}$. If $\gcd(n,10)=1$, let $m=10$, we have $|me|_n+|mc|_n+|m(n-b)|_n+|m(n-a)|_n=n$.

Let $5|n$. If $16c>5n$, since $4n<16b<5n$, $16e<16a<n$, let $m=16$, we have $|me|_n+|mc|_n+|m(n-b)|_n+|m(n-a)|_n=n$.

If $16c<5n$ and $17b<5n$, then $a<\frac{n}{24}$, let $m=17$, we have $|me|_n+|mc|_n+|m(n-b)|_n+|m(n-a)|_n=n$.
If $16c<5n$ and $17b>5n$, then $a<\frac{n}{51}<\frac{b}{15}$, which contradicts to $8a>b$.

{\bf Case 4.} $2<\frac{n}{c}<\frac{n}{b}<3$.

Since $k_1=2$, we have $4<\frac{2n}{c}\leq 5<\frac{2n}{b}<6$, so $m_1 = 5$. Since $\gcd(n, m_1) > 1$, we have $5|n$. The result now follows from Lemma 5.6.
\end{proof}

Now Proposition 2.8 follows immediately from Lemmas 5.1-5.5 and Lemma 5.8.
\vskip1cm
{\noindent\bf Acknowledgements}

The author is thankful to the referees for valuable comments and to prof. Yuanlin Li and
prof. Jiangtao Peng for their useful discussion and valuable suggestions.

\vskip30pt
\def\refname{

\end{document}